\documentclass[12pt,oneside,reqno]{amsart}        

\usepackage{etex}

\usepackage{amssymb,amsthm,amsmath,mathrsfs,comment}
\usepackage{amstext,amsxtra}  
\usepackage{fullpage,colonequals}
\usepackage{bm}       
\usepackage{mathtools} 
\usepackage[all]{xy}
\usepackage{enumitem}

\usepackage[pdfpagelabels]{hyperref}
\hypersetup{pdftitle={Heuristics for narrow class groups and signature ranks},pdfauthor={David S.\ Dummit and John Voight}} 
\hypersetup{colorlinks=true,linkcolor=blue,anchorcolor=blue,citecolor=blue}

\newtheorem{proposition}{Proposition}

\newtheorem*{lemma*}{Lemma} 

\theoremstyle{remark}
\newtheorem*{remark*}{Remark}

\newenvironment{enumalph}
{\begin{enumerate}}
{\end{enumerate}}

\def\ZZ{\mathbb Z}
\def\QQ{\mathbb Q}
\def\OO{\mathcal{O}}   
\def\epsilon{\varepsilon}
\def\phi{\varphi}
\def\F{\mathbb F}
\def\order#1{\vert{#1}\vert}

\DeclareMathOperator{\iso}{\simeq}
\DeclareMathOperator{\Gal}{Gal}
\DeclareMathOperator{\C}{Cl}  
\DeclareMathOperator{\Cplus}{Cl^{\textup{st}}}

\begin{document}

\title[Parity of class numbers and unit signature ranks]
{A Note on the equivalence of the parity of class numbers and the signature ranks of units in cyclotomic fields}

\author{David S.\ Dummit}

\address{Department of Mathematics, University of Vermont, Lord House, 16 Colchester Ave., Burlington, VT 05405, USA}

\email{dummit@math.uvm.edu}

\dedicatory{Dedicated to the memory of my teacher Kenkichi Iwasawa}

\begin{abstract}
We collect some statements regarding equivalence of the parities of various
class numbers and signature ranks of units in prime power cyclotomic fields.  
We correct some misstatements in the literature regarding these parities by providing
an example of a prime cyclotomic field where the signature rank of the units and the signature rank of the
circular units are not equal.  
\end{abstract}

\keywords{class numbers, cyclotomic units}

\subjclass[2010]{11R18 (primary), and 11R27, 11R29 (secondary)} 

\maketitle

\section{Introduction}

Let $p$ be a prime and $n \ge 1$ a fixed integer (with $n \ge 2$ if $p = 2$).  
Let $\zeta_{p^n}$ denote a primitive $p^n$-th root of unity, 
$K = \QQ(\zeta_{p^n})$ the corresponding cyclotomic field of $p^n$-th roots of unity, 
and $K^+ = \QQ( \zeta_{p^n} + \zeta_{p^n}^{-1})$ the maximal totally real subfield of $K$.

\medskip

Denote by $\C(K)$ the class group of $K$, by $\C(K^+)$ the class group of $K^+$, and  
by $\Cplus(K^+)$ the strict (or narrow) class group of $K^+$.

\medskip

Let $E$ denote the group of {\it real} units of $K$, i.e., the units of the maximal real subfield $K^+$ 
(the group of units of $K$ is then $\langle \zeta_{p^n} \rangle \times E$), and let $E^+$ denote the totally
positive units of $K^+$ (or, by abuse, of $K$).

\medskip

Let $C$ denote the subgroup of {\it circular} (or {\it cyclotomic}) {\it units} of $E$ (see \cite[Lemma 8.1]{Wa}),
whose finite index in $E$ is the class number $\order{\C(K^+)}$ (\cite[Theorem 8.2]{Wa}).  Let 
$C^+$ denote the subgroup of totally positive circular units (so $C^+ = C \cap E^+$).

\medskip

The Galois group $\Gal(K/K^+)$ is generated by complex conjugation, which, following Iwasawa, will be
denoted by $J$.  Let $\C^-(K)$, the {\it minus part} of the class group, 
denote the kernel of $1 + J$ acting on $\C(K)$.  Similarly, let $\C^+(K)$ 
denote the kernel of $1 - J$ acting on $\C(K)$.
By class field theory, the class number of $K^+$, $\vert \C(K^+) \vert$,
divides the class number of $K$ and the norm map from $\C(K)$ to $\C(K^+)$ 
is surjective, with kernel $\C^-(K)$, so $\vert \C(K) \vert = \vert \C^-(K) \vert \vert \C(K^+) \vert$.  The
factor $\vert \C^-(K) \vert$ is called the {\it relative} class number of $K$.  
(Warning: the group $\C(K^+)$ embeds into $\C^+(K)$, but $\C^+(K)$ may be strictly larger. 
Classical terminology refers to $\vert \C(K^+) \vert$ (and not
$\vert \C^+(K) \vert$) as the ``plus part" of the class number of $K$ (often as ``$h^+$"), so to 
avoid confusion we shall avoid this terminology.)

\medskip

Equivalencies for the parity of the orders of the various class groups and relations with signature ranks are known
and due to various authors, some beginning as far back as the late 1800's with Kummer and Weber, with
the first systematic study perhaps due to Hasse (\cite{Ha}).  
These equivalencies are summarized in the following proposition.  For the convenience of the reader, 
concise proofs for these equivalencies are given later.

\begin{proposition} \label{prop:equivalences}
With $K = \QQ(\zeta_{p^n})$, $K^+ = \QQ( \zeta_{p^n} + \zeta_{p^n}^{-1})$, and other 
notation as above, the following statements are equivalent:
\begin{enumerate}

\item
The class number of $K$, $\order{ \C(K) } $, is odd.

\item
The {\it relative class number} of $K$, $\order{ \C^-(K) } $, is odd.

\item
The order of $ \C^+(K)$, $\order{ \C^+(K) } $, is odd.

\item
The strict class number of $K^+$, $\order{ \Cplus(K^+) } $, is odd.

\item
One of the following equivalent conditions (a1)--(a3), together with one of the following equivalent conditions (b1)--(b3), holds: 
\begin{enumerate}[leftmargin= .6in]
\item [(a1)] 
the class number of $K^+$, $\order{ \C(K^+) } $, is odd, 
\item [(a2)]
the index $[E:C]$ is odd, 
\item [(a3)]
$C \cap E^2 = C^2$, i.e., every circular unit which is a square (in $K^+$) is the square of a circular unit,
\end{enumerate}
and 
\begin{enumerate}[leftmargin= .65in]
\item [(b1)]
the class number and strict class number of $K^+$ are equal,
\item [(b2)]
every totally positive unit in $E$ is a square in $E$: $E^+ = E^2$, 
\item [(b3)]
there are units of $K^+$ of every possible signature.
\end{enumerate}

\item
There are circular units of $K^+$ of every possible signature (equivalently, every
totally positive circular unit is the square of a circular unit: $C^+ = C^2$).

\end{enumerate}

\end{proposition}

\begin{remark*}
All the statements of the proposition are known to hold when $p = 2$, a result 
due to Weber (\cite[B, p.~821]{We}).  A nice proof of this result by Iwasawa (in the form
of condition (1): ``If $p = 2$,...the class number of $Z_e$ ($e \ge 2$) is...odd") 
can be found in \cite[p.~373]{Iw1}.

\end{remark*}

\begin{remark*}
A number of the implications in the proposition hold for more general fields, with many
of the results in the literature extending to various degrees the results in Hasse's
seminal work \cite{Ha}.  While not exhaustive, particular attention is called to 
the papers by Cornacchia (\cite{C1}, \cite{C2}, \cite{C3}), Garbanati (\cite{Ga1}, \cite{Ga2}), 
G.~Gras and M-N.~Gras (\cite{Gr}, \cite{Gr-Gr}), 
Oriat (\cite{O}),
Stevenhagen (\cite{St}) and the further references they contain.  
\end{remark*}

\medskip

The class number of $K = \QQ( \zeta_{29})$ is 8 and the class number of $K^+ = \QQ( \zeta_{29} + \zeta_{29}^{-1})$ is 1
(\cite[Tables, \S 3, p.~412 and \S 4, p.~421]{Wa},
so for this field the equivalent conditions (a1)--(a3) in (5) are satisfied, but (1) does not hold---hence also
the other statements in Proposition \ref{prop:equivalences} do not hold.  
(It is also known that (6) does not hold by the tables of Davis \cite[p.~70]{Da1}.)
This shows that the equivalent conditions (b1)--(b3) in (5) cannot be dropped (so in particular the
equivalent conditions (a1)--(a3) in (5) do not imply the conditions (b1)--(b3)).  

The purpose of this Note is, in addition to collecting the equivalencies of the proposition above 
in one place, to show (in the following section) that the units in the maximal real subfield of the cyclotomic field of
163rd roots of unity realize all possible signatures but the class number of $\QQ(\zeta_{163})$ is even,
as is the relative class number ``$h_{163}^-$".  Hence for $K = \QQ(\zeta_{163})$, the equivalent conditions (b1)--(b3) in (5)
are satisfied, but the remaining statements in Proposition \ref{prop:equivalences} are not, which shows that
the equivalent conditions (a1)--(a3) in (5) cannot be dropped (so in particular the
equivalent conditions (b1)--(b3) in (5) do not imply the conditions (a1)--(a3)).
This provides a counterexample to  
the assertion that the circular units can be replaced by the full group of real units in equivalence (6) of 
Proposition \ref{prop:equivalences}, an error
that has appeared and propagated in the literature.

\medskip

In \cite{EMP} the authors state that a classical result of Kummer is that every totally positive
unit of $\QQ(\zeta_p + \zeta_p^{-1})$ is a square whenever the class number of $\QQ(\zeta_p)$ is odd (which is part of
the implication (1) implies (5) above), but go on to assert that ``as a result 
of Shimura" (for which they cite \cite{Sh}) ``this is now extended to every totally positive
unit of $\QQ(\zeta_p + \zeta_p^{-1})$ is a square if and only if the class number of $\QQ(\zeta_p)$ is odd." 

In \cite{Es}, the author makes a similar statement that ``With $n$ a prime power, the totally positive
units in $\OO_n^+$ [the integers of the maximal real subfield of the $n$-th roots of unity] are squares of units from $\OO_n^+$
if and only if $h_n^-$ [the relative class number of $\QQ(\zeta_n)$] is odd", citing
Lemma 5 and Theorem 3 in \cite{Ga1}.  

It should be noted that Shimura makes no claim as asserted, in fact stating only that the converse
holds (in a more general setting of imaginary abelian fields of prime power conductor) {\it under 
the additional assumption\/} that the class number of $\QQ(\zeta_p + \zeta_p^{-1})$ is odd
(indicating in a footnote that this was kindly pointed out to him by Iwasawa) \cite[Proposition A.5 and
following, Appendix, p.~83]{Sh}.  Similarly, the link between the signatures of the subgroup of circular units with
the signatures of the full group of units in Lemma 5 of \cite{Ga1} requires the class number of 
$\QQ(\zeta_p + \zeta_p^{-1})$ to be {\it odd\/}.

More recently, this error appears in \cite{K-L}\footnote{There 
is also a gap in the proof of Theorem 3.3 in this paper: the ``$10(r+1)$" in the proof of Lemma 3.2 should
be $10 r + 11$, which need not be even.},
where the authors assert that the ``Taussky conjecture" is that ``every totally 
positive unit of $\QQ(\zeta_q + \zeta_q^{-1})$
is a square" in the case that both $q$ and $p = (q-1)/2$ are primes, stating explicitly that this
is equivalent to the oddness of the class number of $\QQ(\zeta_q)$ (citing \cite{EMP} for the equivalence).
The {\it correct\/} conjecture 
(which as noted by Stevenhagen \cite{St} appears explicitly in print only in the Ph.D.~dissertation and 
subsequent paper of Taussky's 
student Davis (\cite[p.~4]{Da1}, \cite{Da2}) without attribution to Taussky---but note Davis
references \cite{Ta}) is that ``every totally 
positive {\it circular} unit of $\QQ(\zeta_q + \zeta_q^{-1})$
is the square of a {\it circular} unit when $q$ and $p = (q-1)/2$ are both primes".
The terminology ``Taussky's conjecture" in \cite{K-L} is apparently drawn from the discussion in their 
reference \cite{Es}, so either \cite{EMP} and \cite{Es} could be the source of the confusion regarding the
correct conjecture.

\section{The cyclotomic field of 163rd roots of unity}

\begin{proposition}
If $F = \QQ(\zeta_{163})$ is the cyclotomic field of 163rd roots of unity and $F^+ = \QQ(\zeta_{163} +\zeta_{163}^{-1} )$ is
its maximal real subfield, then
\begin{enumalph}
\item
the units of $F^+$ have all possible signatures (i.e., every totally positive unit of $F^+$ is the square of a unit in $F^+$),
while the subgroup of squares of circular units of $F^+$ have index 4 in the group of totally positive circular units of $F^+$ ;

\item
the class number and strict class number of $F^+$ are equal and divisible by 4 (with both equal to 4 under the GRH), 
the power of 2 in the relative class number of $F$ is 4 and the class number of $F$ is divisible by 16 
(with precise 2-power divisor equal to 16 under the GRH).

\end{enumalph}
In particular, every totally positive unit of $F^+$ a square does not imply that the class number (nor, equivalently, the relative class number)
of $F$ is odd.

\end{proposition}

\begin{proof}
The tables of Davis \cite[p.~71]{Da1} show that the rank of the $81 \times 81$ matrix of signatures of the 
circular units in $F^+$ is 79, i.e., $[C: C^+] = 2^{79}$ and $[C:C^2] = 2^{81}$, which gives the second statement in (a).
The relative class number of $F$ given in \cite[Tables, \S 3, p.~415]{Wa} is
$2^2 \cdot 181 \cdot 23167 \cdot 365473 \cdot 441845817162679 $.

Under the GRH the class number of $F^+$ is 4 by \cite{VdL} (see also \cite{Sc}, whose tables are
reproduced in \cite[Tables, \S 4, p.~420]{Wa}).

The field $F^+$ is a cyclic extension of degree 81 over $\QQ$ with cyclic cubic subfield $k^+ = \QQ(\alpha)$ where
$\alpha = \textup{Tr}_{F^+/k} (\zeta_{163} + \zeta_{163}^{-1})$, whose minimal polynomial over $\QQ$ is
$x^3 + x^2 -54 x - 169$.  The class number of $k^+$ is 4 and $\epsilon_1 = \alpha + 4$, $\epsilon_2 =\alpha^2 - 4 \alpha - 34$ are
fundamental units for $k^+$ (\cite[3.3.26569.1]{LMFDB}, but note the database uses $-\alpha$ as generator).  

Since $F^+/k^+$ is totally ramified, it follows that the class number of $F^+$ is divisible by 4 (and equal to 4
under the GRH, as noted above).  Then the class number of $F$, which is the product of the class number of $F^+$ with
the relative class number of $F$, is divisible by 16 (with precise 2-power divisor equal to 16 under the GRH).

It remains to show that the units of $F^+$ have all possible signatures, as this also shows the class number and strict class
number of $F^+$ are the same.  

The units of $F^+$ contain the subgroup $\langle \epsilon_1, \epsilon_2 , C \rangle$ generated by the units of $k^+$ together
with the circular units of $F^+$.  Adding the signatures of $\epsilon_1$ and $\epsilon_2$ as elements of $F^+$ (which are easily computed
since $\alpha$ is a trace) to the signature matrix for $C$
computed as in \cite{Da1} produces a matrix of full rank 81, so the full group of units of $F^+$ also has maximal signature rank,
completing the proof.  
\end{proof}

\begin{remark*}
If the class number of $F^+ =  \QQ(\zeta_{163} +\zeta_{163}^{-1} )$ is indeed equal to 4 as expected, then 
the index of the circular units in the units of $F^+$ is 4.
Since the computation of the rank of the group of signatures shows $C$ has index 4 in $\langle \epsilon_1, \epsilon_2 , C \rangle$,
it would follow that $\langle \epsilon_1, \epsilon_2 , C \rangle$ is the full group of units of $F^+$.
\end{remark*}

\begin{remark*}
The cyclic subfield $k = k^+ (\sqrt{-163})$ of degree 6 contained in $F$
has class group $\C(k)$ isomorphic to $ (\ZZ / 2 \ZZ)^4$ 
(\cite[6.0.115063617043.1]{LMFDB}. 
The class group of $k^+$ is isomorphic to $(\ZZ / 2 \ZZ)^2$ with the cyclic group $\Gal(k^+/\QQ)$ of order 3 acting by its unique
irreducible 2-dimensional representation over $\F_2$ (the finite field of order 2).    
Also, $\C(k)/\C^-(k) \iso \C(k)^{1 + J}$, which by class field theory is isomorphic to $\C(k^+) \iso (\ZZ / 2 \ZZ)^2$.  It follows
that $\C(k)^-$ (which is the same as $\C(k)^+$ since $\C(k)$ has exponent 2) is
isomorphic to $(\ZZ / 2 \ZZ)^2$. 
This implies that $\C(k)$ is isomorphic as a Galois module to the direct sum of two copies of the (unique) 
irreducible 2-dimensional representation of the cyclic group $\Gal(k^+/\QQ)$ of order 3 over $\F_2$,
where $J$ acts by interchanging the two copies.  
Then composing the Hilbert class field of $k$ with $F = \QQ(\zeta_{163})$ shows (under the assumption of the GRH) that the 
Sylow 2-subgroup of $\C(F)$ is isomorphic as a Galois module to the direct sum of two copies of the (unique) 
irreducible 2-dimensional representation of the cyclic group $\Gal(F^+/\QQ)$ of order 81 over $\F_2$,
where $J$ acts by interchanging the two copies.  
\end{remark*}

\section{Proofs of the parity equivalences}

Before giving some concise proofs for the equivalencies in Proposition \ref{prop:equivalences} 
we state a variant of a theorem of Iwasawa \cite{Iw1} that is quite useful (for example, \cite[Proposition 2.1]{N}
is an immediate consequence).

\begin{lemma*}
Suppose $L/F$ is any finite Galois extension of number fields such that 
\begin{enumerate}
\item[(i.)]
the (not necessarily abelian) Galois group $\Gal(L/F)$ has order a power of 2, and

\item[(ii.)]
the extension is unramified outside infinity and a single finite prime where the finite prime is totally ramified.

\end{enumerate}

Then
2 divides the strict class number of $L$ if and only 2 divides the strict class number of $F$. 

\end{lemma*}

\begin{proof}
Note first it suffices to prove the result when $[L:F] = 2$.  Composing the
strict Hilbert class field of $F$ with $L$ gives an extension of the same degree over $L$ that is
unramified at finite primes, so if 2 divides the strict class number of $F$ then 
2 divides the strict class number of $L$. 
Conversely, the strict Hilbert class field $H^{\textup{st}}$ of $L$ is Galois over $F$,
as is the subfield, $H'$, fixed by $2 \Gal(H^{\textup{st}} /L )$, and $H'$ is an elementary abelian 2-extension of
$L$.  Because 2-groups acting on 2-groups necessarily have fixed points, there is a subfield of $H'$
which is an abelian extension
of $F$ of degree 4 containing $L$ as a subfield.  Taking the fixed field of the inertia group for the unique
ramified finite prime in this latter extension gives a quadratic extension of $F$ unramified at
all finite primes, so if 2 divides the strict class number of $L$ then 
2 divides the strict class number of $F$.
\end{proof}


\begin{proof}[Proof of Proposition 1]
Equivalence of (1), (2) and (3): (\cite[p.~576]{Iw2}) 
Let $S(K)$ denote the Sylow 2-subgroup of $\C(K)$, with $S^+(K)$ (the kernel of $1 - J$) and 
minus part $S^-(K)$ (the kernel of $1 + J$).  Then $S^+(K) \cap S^-(K)$ consists of the elements in $S^+(K)$ on
which $1 + J$ acts trivially, i.e., the elements of order 1 or 2 in $S^+(K)$.  Similarly, $S^+(K) \cap S^-(K)$ 
consists of the elements of order 1 or 2 in $S^-(K)$ and it follows that $S^+(K)$ and $S^-(K)$ have the same 2-rank.  In
particular, $S^+(K) = 1$ if and only if $S^-(K) = 1$.  Then $S(K) / S^-(K) = S(K)^{1+J} \subseteq S^+(K)$
shows that $S(K)$ is also trivial if $S^+(K) = S^-(K) = 1$.  
Conversely, $S(K) = 1$ trivially implies  $S^+(K) = S^-(K) = 1$ since $S^+(K)$ and $S^-(K)$ are subgroups of $S(K)$.  
Hence $C(K)$, $C^+(K)$,
and $C^-(K)$ all have the same parity. 

Equivalence of (1) and (4):  
Applying the Lemma to $L = \QQ(\zeta_{p^n})$ and $F = \QQ( \zeta_{p^n} + \zeta_{p^n}^{-1})$ shows that 
2 divides the class number of $K$ (which equals the strict class number as $K$ is complex) if and
only if 2 divides the strict class number of $K^+$.

Conditions (a1) and (a2) are equivalent since $[E:C] = \order{ \C(K^+) }$ (\cite[Theorem 8.2]{Wa}).  To 
see these are equivalent to (a3), note first that $E$ and $C$ are both isomorphic to 
$\ZZ/2\ZZ \times \ZZ^{\phi(p^n)/2 - 1}$ as abelian groups, so 
the groups $E/E^2$ and $C/C^2$ have the same order ($ = 2^{\phi(p^n)/2}$). This together with the isomorphism
$C E^2 / E^2 \iso C/ C \cap E^2$ implies that 
$\order{ E/ C E^2} = [E:E^2] / [C E^2 : E^2] = [C:C^2] /[C : C \cap E^2 ]  = \order{C \cap E^2 / C^2}$.  
Hence $ C \cap E^2 = C^2 $ if and only if $ E = C E^2$.  If  ${\bar E}$ denotes the finite abelian group $E/C$,
then $E = C E^2$ if and only if $\bar E = {\bar E}^2$, which happens if and only if ${\bar E}$ has odd order,
i.e., if and only if $[E:C]$ is odd.  

There are units of $K^+$ of every possible signature if and only if $ [E : E^+] = 2^{\phi(p^n)/2}$, which is
equivalent to $ [ E^+ : E^2] = 1$ since $2^{\phi(p^n)/2} = [E : E^2] = [E : E^+] [ E^+ : E^2]$, so 
conditions (b2) and (b3) are equivalent.  Then $\order{ \Cplus(K^+) } = \order{ \C(K^+) } [E^+ : E^2]$
(for additional details, see \cite[\S 2]{D-V}) 
shows both that (b1) is equivalent to (b2) and that (4) and (5) (in the version (a1) and (b2)) are equivalent.

The two statements in (6) are equivalent since $2^{\phi(p^n)/2} = [C : C^2] = [C : C^+] [ C^+ : C^2]$
so there are circular units of every possible signature (i.e., $[C : C^+] = 2^{\phi(p^n)/2}$) if and only if
$[ C^+ : C^2] = 1$.  Then $C^2 \subseteq C \cap E^2 \subseteq C \cap E^+ = C^+$ shows
that $[ C^+ : C^2] = 1$ if and only if both $C^2 = C \cap E^2$ and $E^+ = E^2$, so (6) is equivalent to
(5) (in the version (a3) and (b2)).  
\end{proof}

\section*{Acknowledgments}
I would like to thank Richard Foote, Hershy Kisilevsky, and Evan Dummit for helpful conversations.


\begin{thebibliography}{LMFDB}
%
%
%
%
\bibitem{C1}
Cornacchia, P.: {\it Anderson's module for cyclotomic fields of prime conductor,
J.~Number Theory}, {\bf 67} (1997), 252--276.
%
\bibitem{C2}
Cornacchia, P.: {\it The parity of the class number of the cyclotomic fields of prime conductor,
Proc. Amer. Math. Soc.}, {\bf 125} (1997), 3163--3168.
%
\bibitem{C3}
Cornacchia, P.: {\it The 2-ideal class groups of $\QQ(\zeta_l)$,
Nagoya Math. J.}, {\bf 162} (2001), 1--18.
%
\bibitem{Da1}
Davis, D. L.: {\it On the distribution of the signs of the conjugates of the cyclotomic units
in the maximal real subfield of the qth cyclotomic fields, q a prime} (Ph.D.~dissertation), 
CalTech,
Pasadena, California, 1969, available from http://thesis.library.caltech.edu/9554/ .
%
\bibitem{Da2}
Davis, D. L.: {\it Computing the number of totally positive circular units which are squares, J.~Number Theory},
{\bf 10} (1978), 1--9.
%
\bibitem{D-V}
Dummit, D.~and Voight, J.: {\it The $2$-Selmer group of a number field and heuristics for narrow class 
groups and signature ranks of units,  Proc.~London Math.~Soc. (3)}, {\bf 117} (2018), 682--726.
%
\bibitem{EMP}
Edgar, H., Mollin, R., and Peterson, B.: {\it  Class groups, totally positive units, and squares, Proc. AMS}, {\bf 98} (1986), 33--37.
%
\bibitem{Es}
Estes, D.R.: {\it On the parity of the class number of the field of $q$-th roots of unity, Rocky Mountain J. Math.},
{\bf 19} (1989), 675--682.
%
\bibitem{Ga1}
Garbanati, D..: {\it Unit signatures, and even class numbers, and relative class numbers, Jour.~f\"ur die Reine und Angewandte Mathematik},
{\bf 274/5} (1975), 376--384.
%
\bibitem{Ga2}
Garbanati, D..: {\it Units with norm $-1$ and signatures of units, Jour.~f\"ur die Reine und Angewandte Mathematik},
{\bf 283/4} (1976), 164--175.
%
%
\bibitem{Gr}
Gras, G.: {\it Parit\'e du nombre de classes et unit\'es cyclotomiques, Asterisque}, {\bf 24/25} (1975), 37--45.
%
\bibitem{Gr-Gr}
Gras, G.~and Gras, M-N.: {\it Signature des unit\'es cyclotomiques et parit\'e du nombre de classes des
extensions cycliques de Q de degr\'e premier impair, Ann.~inst.~Fourier, Grenoble}, {\bf 25} (1975), 1--22.
%
\bibitem{Ha}
Hasse, H.: {\it \"Uber die Klassenzahl abelscher Zahlk\"orper}, Akademie-Verlag, Berlin, 1952.
%
\bibitem{H-M}
Hughes, I.~and Mollin, R.: {\it Totally positive units and squares,
Proc. Amer. Math. Soc.}, {\bf 87} (1983), no. 4, 613--616.
%
\bibitem{Iw1}
Iwasawa, K.: {\it A note on class numbers of algebraic number fields}. In Ichiro Satake et al. (Eds.), 
Kenkichi Iwasawa Collected Papers, 
(Vol.~I, pp.~372-373), Springer-Verlag, Tokyo, 2001 (original work published 1956).
%
%
\bibitem{Iw2}
Iwasawa, K.: {\it A note on ideal class groups}. In Ichiro Satake et al. (Eds.), 
Kenkichi Iwasawa Collected Papers, 
(Vol.~II, pp.~239-247), Springer-Verlag, Tokyo, 2001 (original work published 1966).
%
\bibitem{K-L}
Kim, M-H., and Lim, S-G.: {\it Square classes of totally positive units, J.~Number Theory},
{\bf 125} (2007), 1--6.
%
%
\bibitem{LMFDB}
The LMFDB Collaboration: {\it The L-functions and Modular Forms Database, http://www.lmfdb.org}, 2013, 
[Online; accessed November 2017].
%
\bibitem{N}
Neumann, O.: {\it On maximal p-extensions, class numbers and unit signatures, Journ\'ees Arithm\'etiques de Caen 
(Univ. Caen, Caen, 1976), Asterisque No. 41-42}, Soc. Math. France, Paris, 1977, 239--246.
%
\bibitem{O}
Oriat, B.: {\it Relation entre les $2$-groupes de classes d'id\'eaux au sens ordinaire et restreint de certains
corps de nombres, Bull. Soc. Math. France} {\bf 104} (1976), 301-307.
%
%
\bibitem{Sc}
Schoof, R.: {\it Class numbers of real cyclotomic fields of prime conductor, Math.~Comp.}, {\bf 72} (2003), 913--937.
%
%
\bibitem{Sh}
Shimura, G.: {\it On abelian varieties with complex multiplication, Proc.~LMS (3)}, {\bf 34} (1977), 65--86.
%
%
\bibitem{St}
Stevenhagen, P.: {\it Class number parity for the $p$th cyclotomic field, Math.~Comp.}, {\bf 63} (1994), 773--784.
%
%
\bibitem{Ta}
Taussky, O.: {\it Unimodular integral circulants, Math.~Zeitschrift}, {\bf 63} (1955), 286--298.
%
%
\bibitem{VdL}
Van der Linden, F.: {\it Class number computations of real abelian number fields, Math.~Comp.}, {\bf 39} (1982), 693--707.
%
%
\bibitem{Wa}
Washington, L.: {\it Introduction to Cyclotomic Fields, Second Edition}, Springer-Verlag, New York, 1997.
%
%
\bibitem{We}
Weber, H.: {\it Lehrbuch der Algebra, vol. II, Third Edition}, AMS Chelsea Publishing, Providence, Rhode Island, 2000 (reprint of
1899 Second Edition).
%
%
\end{thebibliography}
\end{document}